\documentclass{article}

\usepackage{amssymb,amsmath,color}
\usepackage{amsthm}
\usepackage{url}

\definecolor{red}{rgb}{1,0,0}
\definecolor{blue}{rgb}{.2,.2,.8}

\newtheorem{theorem}{Theorem}[section]
\newtheorem{corollary}[theorem]{Corollary}

\theoremstyle{definition}

\begin{document}

\title{Combinatorial interpretations of two identities of Guo and Yang}
\author{Mircea Merca\footnote{mircea.merca@profinfo.edu.ro}
	\\ 
	%\small University of Craiova\\
	\small Department of Mathematics, University of Craiova, 200585 Craiova, Romania\\
	\small Academy of Romanian Scientists, Ilfov 3, Sector 5, Bucharest, Romania
	%\small Department of Informatics,
	%\small Nicolae Grigorescu National College,
	%\small Campina, 100600 Romania
}
\date{}
\maketitle
%\ccom{} \mcom{}

\begin{abstract}
  The restricted partitions in which the largest part is less than or equal to $N$ and the number of parts is less than or equal to $k$ were investigated by Andrews in \cite{Andrews76}. These partitions were extended recently by the author to the partitions into parts of two kinds. In this paper, we use a new class of restricted partitions into parts of two kinds to provide new combinatorial interpretations for two identities of Guo and Yang.
\end{abstract}

\section{Introduction}

A partition of $n$ into at most $k$ parts, each part less than or equal to $N$ is an unordered sum of $n$ that uses at most $k$ positive integers less than or equal to $N$. 
These partitions were investigated by Andrews in \cite{Andrews76}. 
Following the notation in \cite{Andrews76}, the number of such partitions will be denoted in this paper by $p(N,k,n)$. 
According to \cite[Theorem 3.1]{Andrews76}, the generating function of $p(N,k,n)$ is given by
\begin{equation}\label{a1}
\sum_{n=0}^{Nk} p(N,k,n) q^n = \begin{bmatrix}N+k\\N\end{bmatrix}_q,
\end{equation}
where
$$
\begin{bmatrix}n\\k\end{bmatrix}_q=
\begin{cases}
0, & \text{if $k<0$ or $k>n$,} \\
\dfrac{(q;q)_n}{(q;q)_k(q;q)_{n-k}}, & \text{otherwise}
\end{cases}
$$
is the Gaussian polynomial or the $q$-binomial coefficient. Recall that
$$(a;q)_{n}=
\begin{cases}
1,  &\text{for $n=0$,}\\
(1-a)(1-aq)(1-aq^2)\cdots(1-aq^{n-1}), &\text{for $n>0$}
\end{cases}
$$
is the $q$-shifted factorial and 
$$(a;q)_{\infty} = \lim_{n\to \infty} (a;q)_n.$$
Because the infinite product $(a;q)_{\infty}$ diverges when $a\neq 0$ and $|q|\geqslant 1$, whenever $(a;q)_{\infty}$ appears in a formula, we shall assume that $|q|<1$.	

Assume there are positive integers of two kinds: $\lambda$ and $\overline{\lambda}$.
We denote by $\overline{p}_r(N_1,N_2,k_1,k_2,n)$ the number of partitions of $n$ into parts of two kinds with
at most $k_1$ parts of the first kind, each part divisible by $r$ and less than or equal to $N_1r$, and
at most $k_2$ parts of the second kind, each part less than or equal to $N_2$. 
For example, $\overline{p}_2(2,3,2,2,4)=6$ because the six partitions in question are:
$$4,\quad 2+2,\quad  2+\overline{2},\quad 2+\overline{1}+\overline{1},\quad \overline{3}+\overline{1}\quad \text{and} \quad \overline{2}+ \overline{2}.$$

Recently, Merca \cite{Merca} considered some properties of the Gaussian polynomials and obtained few properties of $\overline{p}_1(N_1,N_2,k_1,k_2,n)$ (see for instance \cite[Theorems 3.1, 3.3, 3.4, 3.5, 4.1, 4.5, 5.1, 5.2, 6.1]{Merca}). In this paper, motivated by these results, we provide some similar results for $\overline{p}_r(N_1,N_2,k_1,k_2,n)$.
Two partition formulas involving $\overline{p}_2(N_1,N_2,k_1,k_2,n)$ and $\overline{p}_4(N_1,N_2,k_1,k_2,n)$ are derived in the last section of this paper as corollaries of two identities of Guo and Yang \cite{GuoYang}:
\begin{align}
& \sum_{k=0}^{\lfloor n/2 \rfloor} \begin{bmatrix}m+k\\k\end{bmatrix}_{q^2} \begin{bmatrix}m+1\\n-2k\end{bmatrix}_{q} q^{\binom{n-2k}{2}} 
 = \begin{bmatrix}m+n\\n\end{bmatrix}_{q}, \label{e2}\\
& \sum_{k=0}^{\lfloor n/4 \rfloor} \begin{bmatrix}m+k\\k\end{bmatrix}_{q^4} \begin{bmatrix}m+1\\n-4k\end{bmatrix}_{q} q^{\binom{n-4k}{2}} \notag \\
& \qquad\qquad\qquad = \sum_{k=0}^{\lfloor n/2 \rfloor} (-1)^k \begin{bmatrix}m+k\\k\end{bmatrix}_{q^2} \begin{bmatrix}m+n-2k\\n-2k\end{bmatrix}_{q}.\label{e3}
\end{align}

We remark that these identities are specializations of two convolutions for the complete and elementary symmetric functions (see \cite[Theorems 2.1 and 2.5]{Merca1}).

\section{Some general results}

The following convolution is a connection between the partition functions $p(N,k,n)$ and $\overline{p}_r(N_1,N_2,k_1,k_2,n)$.

\begin{theorem}\label{T2.1}
	For $N_1,N_2,k_1,k_2,n\geqslant 0$, $r\geqslant 1$,
	$$\overline{p}_r(N_1,N_2,k_1,k_2,n) = \sum_{j=0}^{\lfloor n/r \rfloor} p(N_1,k_1,j)p(N_2,k_2,n-rj).$$
\end{theorem}

\begin{proof}
	Let
	$$
	\lambda_1+\cdots+\lambda_a+\overline{\lambda}_1+\cdots+\overline{\lambda}_b = n
	$$
	be a partition into parts of two kinds with $\lambda_i\leqslant N_1r$, $\overline{\lambda}_i\leqslant N_2$, $\lambda_i$ divisible by $r$, $0 \leqslant a \leqslant k_1$ and $0 \leqslant b \leqslant k_2$. 
    This partition can be rewritten as
	 $$
	 r\sum_{i=1}^{a}\lambda'_i+\sum_{i=1}^{b}\overline{\lambda}_i=n,
	 $$
	 where $\lambda'_i = \lambda / r$ and $\lambda'_i \leqslant N_1$.
	 The identity follows easily from this relation. 
\end{proof}

This convolution allows us to give the generating function for $\overline{p}_r(N_1,N_2,k_1,k_2,n)$.

\begin{theorem}\label{T2.2}
	For $N_1,N_2,k_1,k_2\geqslant 0$, $r\geqslant 1$,
	$$\sum_{n=0}^{N_1k_1r+N_2k_2} \overline{p}_r(N_1,N_2,k_1,k_2,n) q^n = \begin{bmatrix}N_1+k_1\\N_1\end{bmatrix}_{q^r} \begin{bmatrix}N_2+k_2\\N_2\end{bmatrix}_q.$$		
\end{theorem}

\begin{proof}
	Taking into account \eqref{a1}, we can write
	\begin{align*}
	& \begin{bmatrix}N_1+k_1\\N_1\end{bmatrix}_{q^r} \begin{bmatrix}N_2+k_2\\N_2\end{bmatrix}_q \\
	& \qquad = \left( \sum_{n=0}^{N_1k_1} p(N_1,k_1,n) q^{nr} \right) \left( \sum_{n=0}^{N_2k_2} p(N_2,k_2,n) q^n \right) \\
	& \qquad = \sum_{n=0}^{\infty} \left( \sum_{j=0}^{\lfloor n/r \rfloor} p(N_1,k_1,j)p(N_2,k_2,n-rj) \right) q^n,
	\end{align*}
	where we have invoked the Cauchy multiplication of two power series. The proof follows considering Theorem \ref{T2.1}.
\end{proof}

The recurrence relations for the Gaussian polynomials 
\begin{equation}\label{GR1}
\begin{bmatrix}N\\k\end{bmatrix}_{q^r} = q^{kr} \begin{bmatrix}N-1\\k\end{bmatrix}_{q^r} + \begin{bmatrix}N-1\\k-1\end{bmatrix}_{q^r}.
\end{equation}
and
\begin{equation}\label{GR2}
\begin{bmatrix}N\\k\end{bmatrix}_{q^r} = \begin{bmatrix}N-1\\k\end{bmatrix}_{q^r} +q^{(N-k)r} \begin{bmatrix}N-1\\k-1\end{bmatrix}_{q^r}.
\end{equation}
can be used to derive the following generalization of \cite[Theorem 3.3]{Merca}.

\begin{theorem}\label{2.3}
		For $N_1,N_2,k_1,k_2,n\geqslant 0$, $r\geqslant 1$,
		\begin{enumerate}
			\item $\overline{p}_r(N_1,N_2,k_1,k_2,n) - \overline{p}_r(N_1-1,N_2-1,k_1,k_2,n-k_1r-k_2)   $
			\item[] $\quad   - \overline{p}_r(N_1-1,N_2,k_1,k_2-1,n-k_1r) - \overline{p}_r(N_1,N_2-1,k_1-1,k_2,n-k_2)  $
			\item[] $\quad    - \overline{p}_r(N_1,N_2,k_1-1,k_2-1,n) = 0;$
			\item $\overline{p}_r(N_1,N_2,k_1,k_2,n) - \overline{p}_r(N_1-1,N_2-1,k_1,k_2,n) $
			\item[] $\quad    - \overline{p}_r(N_1-1,N_2,k_1,k_2-1,n-N_2) - \overline{p}_r(N_1,N_2-1,k_1-1,k_2,n-N_1r) $
			\item[] $\quad   - \overline{p}_r(N_1,N_2,k_1-1,k_2-1,n-N_1r-N_2) = 0;$
			\item $\overline{p}_r(N_1,N_2,k_1,k_2,n) - \overline{p}_r(N_1-1,N_2-1,k_1,k_2,n-k_1r)  $
			\item[] $\quad - \overline{p}_r(N_1-1,N_2,k_1,k_2-1,n-k_1r-N_2) - \overline{p}_r(N_1,N_2-1,k_1-1,k_2,n) $
			\item[] $\quad  - \overline{p}_r(N_1,N_2,k_1-1,k_2-1,n-N_2) = 0.$
			\end{enumerate}
\end{theorem}

\begin{proof}
		By \eqref{GR1}, we have
		\begin{align*}
		& \begin{bmatrix}N_1+k_1\\k_1\end{bmatrix}_{q^r}\begin{bmatrix}N_2+k_2\\k_2\end{bmatrix}_{q}	\\
		& \quad = \left( q^{k_1r} \begin{bmatrix}N_1-1+k_1\\k_1\end{bmatrix}_{q^r} + \begin{bmatrix}N_1+k_1-1\\k_1-1\end{bmatrix}_{q^r}\right) 
		\left( q^{k_2} \begin{bmatrix}N_2-1+k_2\\k_2\end{bmatrix}_q + \begin{bmatrix}N_2+k_2-1\\k_2-1\end{bmatrix}_q \right) \\
		& \quad = q^{k_1r+k_2} \begin{bmatrix}N_1-1+k_1\\k_1\end{bmatrix}_{q^r} \begin{bmatrix}N_2-1+k_2\\k_2\end{bmatrix}_{q} 
		+ q^{k_1r} \begin{bmatrix}N_1-1+k_1\\k_1\end{bmatrix}_{q^r} \begin{bmatrix}N_2+k_2-1\\k_2-1\end{bmatrix}_{q} \\
		& \qquad\qquad    + q^{k_2} \begin{bmatrix}N_1+k_1-1\\k_1-1\end{bmatrix}_{q^r} \begin{bmatrix}N_2-1+k_2\\k_2\end{bmatrix}_{q}
		+ \begin{bmatrix}N_1+k_1-1\\k_1-1\end{bmatrix}_{q^r} \begin{bmatrix}N_2+k_2-1\\k_2-1\end{bmatrix}_{q}.
		\end{align*}
		This allows us to write
		\begin{enumerate}
			\item[] $\displaystyle{\sum_{n=0}^{N_1k_1r+N_2k_2} \overline{p}_r(N_1,N_2,k_1,k_2,n) q^n}$
			\item[] $\displaystyle{\qquad\quad = \sum_{n=0}^{(N_1-1)k_1r+(N_2-1)k_2} \overline{p}_r(N_1-1,N_2-1,k_1,k_2,n) q^{n+k_1r+k_2}}$
			\item[] $\displaystyle{\qquad\qquad + \sum_{n=0}^{(N_1-1)k_1r+N_2(k_2-1)} \overline{p}_r(N_1-1,N_2,k_1,k_2-1,n) q^{n+k_1r}}$
			\item[] $\displaystyle{\qquad\qquad + \sum_{n=0}^{N_1(k_1-1)r+(N_2-1)k_2} \overline{p}_r(N_1,N_2-1,k_1-1,k_2,n) q^{n+k_2}}$
			\item[] $\displaystyle{\qquad\qquad + \sum_{n=0}^{N_1(k_1-1)r+N_2(k_2-1)} \overline{p}_r(N_1,N_2,k_1-1,k_2-1,n) q^{n}.}$
		\end{enumerate}
		The proof of the first relation follows equating the coefficient of $q^n$ in this identity. Similarly, considering \eqref{GR2} we obtain the second recurrence relation. The last relation follows combining \eqref{GR1} and \eqref{GR2}.
\end{proof} 

For $r\geqslant 1$, it is clear that the Gaussian polynomial 
$$\begin{bmatrix}N+k\\N\end{bmatrix}_{q^r}$$
is symmetric in $N$ and $k$. In addition, this polynomial is self-reciprocal. These properties allow us to derive the following relations for the partition function $\overline{p}_r(N_1,N_2,k_1,k_2,n)$.

\begin{theorem}\label{T2.4}
		For all $N_1,N_2,k,n \geqslant 0$, $r\geqslant 1$,
		\begin{align*}
		& \overline{p}_r(N_1,N_2,k_1,k_2,n) = \overline{p}_r(k_1,N_2,N_1,k_2,n) = \overline{p}_r(N_1,k_2,k_1,N_2,n) = \overline{p}_r(k_1,k_2,N_1,N_2,n);\\
		& \overline{p}_r(N_1,N_2,k_1,k_2,n) = \overline{p}_r(N_1,N_2,k_1,k_2,N_1k_1r+N_2k_2-n).
		\end{align*}
\end{theorem}

\begin{proof}
	The poof is similar to \cite[Theorem 3.4]{Merca}.
\end{proof}

For $r>1$, we remark that the Gaussian polynomial 
$$\begin{bmatrix}N+k\\N\end{bmatrix}_{q^r}$$
is not unimodal. Thus the third relation of \cite[Theorem 3.4]{Merca} can not be generalized in this way.

We denote by $\overline{Q}_r(N_1,N_2,k_1,k_2,n)$ the number of partitions of $n$ into parts of two kinds with
exactly $k_1$ distinct parts of the first kind, each part divisible by $r$ and less than or equal to $N_1$ and
exactly $k_2$ distinct parts of the second kind, each part less than or equal to $N_2$.
We have the following bijection between restricted partitions into parts of two kinds.

\begin{theorem}\label{T2.5}
	For $N_1,N_2,k_1,k_2,n\geqslant 0$, $r\geqslant 1$,
	$$\overline{Q}_r(N_1,N_2,k_1,k_2,n) = \overline{p}_r\Bigg( N_1-k_1,N_2-k_2,k_1,k_2,n-r\binom{k_1+1}{2}-\binom{k_2+1}{2}\Bigg).$$
\end{theorem}

\begin{proof}
	The poof is similar to \cite[Theorem 6.1]{Merca}.
\end{proof}

The generating function for $\overline{Q}_r(N_1,N_2,k_1,k_2,n)$ can be easily derived from Theorems \ref{T2.2} and \ref{T2.5}.

\begin{theorem}\label{T2.6}
	For $N_1,N_2,k_1,k_2\geqslant 0$, $r\geqslant 1$,
	$$\sum_{n=0}^{\infty} \overline{Q}_r(N_1,N_2,k_1,k_2,n) q^n = q^{r\binom{k_1+1}{2}+\binom{k_2+1}{2}}\begin{bmatrix}N_1\\k_1\end{bmatrix}_{q^r} \begin{bmatrix}N_2\\k_2\end{bmatrix}_q.$$		
\end{theorem}

A similar result to Theorem \ref{T2.1} is also possible for $\overline{Q}_r(N_1,N_2,k_1,k_2,n)$ if we consider the number of partitions of $n$ into exactly $k$ distinct parts, each part less than or equal to $N$ which is denoted in \cite{Merca} by $Q(N,k,n)$. In fact, all results obtained for  $\overline{p}_r(N_1,N_2,k_1,k_2,n)$ can be rewritten in terms of the partition functions $\overline{Q}_r(N_1,N_2,k_1,k_2,n)$.

\section{Two partition formulas}

As we can see in \cite[Theorems 4.5 and 5.1]{Merca}, the partition function $p(N,k,n)$ can be expressed in terms of the partition function $\overline{p}_1(N_1,N_2,k_1,k_2,n)$. For instance, the identity
$$p(N,k,n) = \sum_{j=0}^{k} \overline{p}_1(N-j,k-j,j,j,n-j^2)$$
is a specialization of \cite[Theorem 5.1]{Merca}.
The following result shows that the partition function $p(N,k,n)$ can be expressed in terms of the partition function $\overline{p}_2(N_1,N_2,k_1,k_2,n)$.

\begin{theorem}\label{T3.1}
	For $N,k,n\geqslant 0$,
	$$p(N,k,n) = \sum_{j=0}^{\lfloor k/2 \rfloor} \overline{p}_2\Bigg( N,N+1-k+2j,j,k-2j,n-\binom{k-2j}{2}\Bigg).$$
\end{theorem}

\begin{proof}
	Taking into account the first identity of Guo and Yang \eqref{e2}, we can write
	\begin{align*}
	& \sum_{n=0}^{Nk} p(N,k,n) q^n \\
	& \qquad= \begin{bmatrix}N+k\\k\end{bmatrix}_q \\
	& \qquad= \sum_{j=0}^{\lfloor k/2 \rfloor} \begin{bmatrix}N+j\\j\end{bmatrix}_{q^2} \begin{bmatrix}N+1\\k-2j\end{bmatrix}_{q} q^{\binom{k-2j}{2}} \\
	& \qquad= \sum_{j=0}^{\lfloor k/2 \rfloor} \sum_{n=0}^{2Nj+(N+1-k+2j)(k-2j)} \overline{p}_2(N,N+1-k+2j,j,k-2j,n) q^{n+\binom{k-2j}{2}}\\
	& \qquad= \sum_{j=0}^{\lfloor k/2 \rfloor} \sum_{n=\binom{k-2j}{2}}^{Nk-\binom{k-2j}{2}} \overline{p}_2\Bigg( N,N+1-k+2j,j,k-2j,n-\binom{k-2j}{2}\Bigg)  q^{n}\\	
	& \qquad= \sum_{n=0}^{Nk} \sum_{j=0}^{\lfloor k/2 \rfloor}  \overline{p}_2\Bigg( N,N+1-k+2j,j,k-2j,n-\binom{k-2j}{2}\Bigg)  q^{n}.
	\end{align*} 
	The identity follows equating the coefficient of $q^n$ in this relation.
\end{proof}

As a consequence of this theorem, we remark the following formula for the partition function $p(n)$. 

\begin{corollary}\label{C3.2}
	For $n\geqslant 0$,
	$$p(n) = \sum_{j=\left\lceil \frac{n}{2}-\frac{1}{4}-\sqrt{\frac{n}{2}+\frac{1}{16}}\right\rceil}^{\lfloor n/2 \rfloor} \overline{p}_2\Bigg( n,n-2j,j,2j+1,n-\binom{n-2j}{2}\Bigg).$$
\end{corollary}

It is clear that the expansion of $p(n)$ by this corollary requires about $1+\sqrt{n/2}$ terms. For example, the case $n=6$ of this corollary read as
$$p(6)=\overline{p}_2(6,4,1,3,0)+\overline{p}_2(6,2,2,5,5)+\overline{p}_2(6,0,3,7,6)=1+7+3=11.$$ 

The partition functions $\overline{p}_2(N_1,N_2,k_1,k_2,n)$ and $\overline{p}_4(N_1,N_2,k_1,k_2,n)$ are related by the following identity.

\begin{theorem}\label{T3.3}
	For $N,k,n\geqslant 0$,
	$$\sum_{j=0}^{\lfloor k/4 \rfloor} \overline{p}_4\Bigg( N,N+1-k+4j,j,k-4j,n-\binom{k-4j}{2}\Bigg)
	=\sum_{j=0}^{\lfloor k/2 \rfloor} (-1)^j \overline{p}_2\left( N,N,j,k-2j,n\right).$$
\end{theorem}

\begin{proof}
	We have
	\begin{align*}
	& \sum_{j=0}^{\lfloor k/4 \rfloor} \begin{bmatrix}N+j\\j\end{bmatrix}_{q^4} \begin{bmatrix}N+1\\k-4j\end{bmatrix}_q q^{\binom{k-4j}{2}}\\
	& \qquad = \sum_{j=0}^{\lfloor k/4 \rfloor} \sum_{n=0}^{4Nj+(N+1-k+4j)(k-4j)} \overline{p}_4(N,N+1-k+4j,j,k-4j,n) q^{n+\binom{k-4j}{2}}\\
	& \qquad = \sum_{j=0}^{\lfloor k/4 \rfloor} \sum_{n=\binom{k-4j}{2}}^{Nk-\binom{k-4j}{2}} \overline{p}_4\Bigg( N,N+1-k+4j,j,k-4j,n-\binom{k-4j}{2}\Bigg)  q^{n}\\
	& \qquad = \sum_{n=0}^{Nk} \sum_{j=0}^{\lfloor k/4 \rfloor}  \overline{p}_4\Bigg( N,N+1-k+4j,j,k-4j,n-\binom{k-4j}{2}\Bigg)  q^{n}	
	\end{align*}
	and
	\begin{align*}
	& \sum_{j=0}^{\lfloor k/2 \rfloor} (-1)^j \begin{bmatrix}N+j\\j\end{bmatrix}_{q^2} \begin{bmatrix}N+k-2j\\k-2j\end{bmatrix}_q \\
	& \qquad = \sum_{j=0}^{\lfloor k/2 \rfloor} \sum_{n=0}^{2Nj+N(k-2j)} (-1)^j \overline{p}_2(N,N,j,k-2j,n) q^{n}\\
	& \qquad = \sum_{n=0}^{Nk} \sum_{j=0}^{\lfloor k/2 \rfloor}  \overline{p}_2\left( N,N,j,k-2j,n\right)  q^{n}.
	\end{align*}
	The identity follows easily considering the identity \eqref{e3}.
\end{proof}

Theorems \ref{T3.1} and \ref{T3.3} provide new combinatorial interpretations for the  identities \eqref{e2} and \eqref{e3} of Guo and Yang.

%Finally, we remark that Theorem \ref{T5} provides an extremely efficient algorithm for computing the partition function $P_2(n)$.

%\section*{Acknowledgements} The second author wishes to thank the College of the Holy Cross for its hospitality during the writing of this article.



\bigskip

%\noindent\textit{Department of Mathematics,
%University of Craiova, 200585 Craiova, Romania\\
%mircea.merca@profinfo.edu.ro}

\end{document}